\documentclass[12pt]{article}

\usepackage{amsfonts,amsmath,amssymb,amsthm,geometry}

\newtheorem{theorem}{Theorem}[section]
\newtheorem{lemma}{Lemma}[section]
\newtheorem{corollary}{Corollary}[section]
\newtheorem{proposition}{Proposition}[section]

\title{Finite convexification by Minkowski summation}
\author{Mark Meyer}

\begin{document}

\maketitle

\begin{abstract}
    Let $A\subset\mathbb{R}^n$ be a nonempty compact set and let $K\subset\mathbb{R}^n$ be a nonempty compact convex set. We study the problem of when there exists $\lambda>0$ such that the Minkowski sum $A+\lambda K$ is convex. First, we show that this finite convexification is equivalent to a uniform covering condition on the exposed sets of $A$ and $K$. We then show that another characterization is possible in terms of the covering radii of the iterated exposed sets after translating each exposed face of $K$ so that its relative centroid is $0$. In the case of convex polytopes, we show that it suffices to check the exposed sets in the first level of the iteration. We also study several special cases, such as strictly convex convexifiers, lower-dimensional convexifiers, strongly isomorphic polytopes, and the planar case. 
\end{abstract}

\section{Introduction}

Throughout this paper, we let $n\geq 2$ be an integer. The Minkowski sum of nonempty compact sets $A,B\subset\mathbb{R}^n$ is defined by $A+B:=\{a+b:a\in A,b\in B\}$. There has been a significant amount of work done to understand the convexification effect of Minkowski summation, which is the phenomenon that the $m$-fold sum $\sum_{i=1}^{m}A$ of a nonconvex set $A$ closely resembles its convex hull with respect to certain measures of nonconvexity as $m\rightarrow \infty$, possibly after scaling by an appropriate constant. It is reasonable to study this problem from different perspectives, such as when the sets in the sum differ from each other, and when some of the sets are convex. We will study one such perspective. Given a nonempty compact convex $K\subset\mathbb{R}^n$ and a nonempty compact $A\subset\mathbb{R}^n$, we define
\begin{align*}
    c^{(K)}(A):=\inf\{\lambda>0:A+\lambda K\textup{ is convex\}}
\end{align*}
to be the infimum of all scalings $\lambda>0$ for which $\lambda K:=\{\lambda k:k\in K\}$ convexifies $A$ by Minkowski summation. The quantity $c^{(K)}(A)$ acts as a threshold for when the sum $A+\lambda K$ is convex or nonconvex: if $c^{(K)}(A)<\infty$, then choose $\lambda>c^{(K)}(A)$. By definition, there exists $c^{(K)}(A)\leq r<\lambda$ such that $A+rK$ is convex. Then 
\begin{align*}
    A+\lambda K=(A+rK)+(\lambda-r)K,
\end{align*}
which is a convex set since both sets on the right side are convex. Note that in the case $K=\textup{conv}(A)$, we obtain the nonconvexity index defined by Schneider in \cite{schneider1}.

We observe that $c^{(K)}(A)$ is not necessarily finite. For example, take $A:=\{(0,0),(1,0)\}$, and take $K:=B_2^2$ to be the unit $\ell_2$-ball of $\mathbb{R}^2$. Then $c^{(K)}(A)=\infty$. The natural problem is therefore to find conditions on the pair $(A,K)$ for which $c^{(K)}(A)$ is finite. Our main results show that the finiteness of $c^{(K)}(A)$ is equivalent to a uniform covering property of the exposed sets of $A$ and $K$. The results are formulated as follows.

For a nonempty compact convex $G\subset\mathbb{R}^n$ with $0\in \textup{relint}(G)$, and for a nonempty compact $B\subset\mathbb{R}^n$, define the $G$-covering radius of $B$ by 
\begin{align*}
    \rho^{(G)}(B):=\inf\{r>0:\textup{conv}(B)\subset B+rG\}.
\end{align*}
Note that $\rho^{(G)}(B)$ may not be finite. We discuss this aspect of the problem in Section \ref{subsection:lower_dimensional_convexifiers}. Let $\mathbb{S}^{n-1}$ denote the unit sphere of $\mathbb{R}^n$. For a nonempty compact set $S\subset\mathbb{R}^n$ and direction $u\in \mathbb{S}^{n-1}$, define
\begin{align*}
    S_u:=\{x\in S:\langle x,u\rangle=h_S(u)\}
\end{align*}
to be the exposed set of $S$ in direction $u$, where $h_S(u):=\sup_{x\in S}\langle x,u\rangle$ is the support function of $S$. The first of our main results characterizes finiteness of $c^{(K)}(A)$ according to the universal covering $K_u$-radii
\begin{align*}
    \Delta^{(K_u)}(A_u):=\sup_{p\in \textup{relint}(K_u)}\rho^{(K_u-p)}(A_u),\qquad u\in \mathbb{S}^{n-1}
\end{align*}
of the exposed sets $A_u$.

\begin{theorem}\label{theorem:finite_condition_universal_translate}
    Let $K\subset\mathbb{R}^n$ be a compact convex set with $\textup{dim}(K)\geq 1$, and let $A\subset\mathbb{R}^n$ be a nonempty compact set. Then
    \begin{align*}
        c^{(K)}(A)<\infty \iff \sup_{u\in \mathbb{S}^{n-1}}\Delta^{(K_{u})}(A_u)<\infty.
    \end{align*}
\end{theorem}

In other words, to check that $A+\lambda K$ is convex for sufficiently large $\lambda$, it suffices to study the exposed faces of $A$ and $K$, and moreover, we only need to determine the covering condition $A_u+r(K_u-p)\supset \textup{conv}(A_u)$ for a bounded $r$, and it is not necessary to know whether $A_u+r(K_u-p)$ is convex.

It is natural to ask whether it is possible to strengthen Theorem \ref{theorem:finite_condition_universal_translate} by replacing the uniform boundedness of the uniform covering radii $\Delta^{(K_u)}(A_u)$ with a uniform boundedness condition on the covering radii $\rho^{(K_u-p_u)}(A_u)$ for fixed translates $p_u\in \textup{relint}(K_u)$. We obtain a statement of this type that depends on the iterated covering radii, defined as follows. For $2\leq j\leq n-1$, let $O_j$ be the collection of all $j$-tuples $(u_1,\dots,u_j)$ of vectors $u_i\in \mathbb{S}^{n-1}$ that are pairwise orthogonal. If $j=1$, then set $O_1:=\mathbb{S}^{n-1}$. For a set $S\subset \mathbb{R}^n$ and $\textbf{u}\in O_j$, the iterated exposed set of $S$ is defined inductively by
\begin{align*}
    S_{\textbf{u}}:=(S_{(u_1,\dots,u_{j-1})})_{u_j}.
\end{align*}
The relative centroid of a nonempty compact convex $G\subset\mathbb{R}^n$ is defined by
\begin{align*}
    g(G):=\frac{1}{|G|_k}\int_{G}x\hspace{0.5mm}dx,\qquad k=\textup{dim}(G),
\end{align*}
 where $|\cdot |_k$ is $k$-dimensional volume in $\textup{aff}(G)$. We denote $G^*:=G-g(G)$ to be the translate of $G$ whose relative centroid is $0$. With this notation, we prove the following centroid version of Theorem \ref{theorem:finite_condition_universal_translate}.

\begin{theorem}\label{theorem:main_theorem}
    Let $K\subset\mathbb{R}^n$ be a nonempty compact convex set with $\textup{dim}(K)\geq 1$, and let $A\subset\mathbb{R}^n$ be nonempty and compact. Then
    \begin{align}\label{eq:covering_radius_finite_condition}
        c^{(K)}(A)<\infty \iff \sup_{u\in O_j}\rho^{(K_{\textbf{u}}^*)}(A_{\textbf{u}})<\infty \textup{ for all } 1\leq j\leq n-1.
    \end{align}
\end{theorem}

 The uniform boundedness is essential, as we show in Section \ref{subsection:uniform_boundedness_essential}. We do not know if it is necessary to check condition \eqref{eq:covering_radius_finite_condition} for all $j$, or if it suffices to verify this condition for $j=1$. We are able to show that $j=1$ suffices in the case where $K$ and $\textup{conv}(A)$ are convex polytopes.

\begin{theorem}\label{theorem:polygone_equivalence}
    Let $K\subset\mathbb{R}^n$ be a compact convex set with $\textup{dim}(K)\geq 1$, and let $A\subset\mathbb{R}^n$ be a nonempty compact set. If $K$ and $\textup{conv}(A)$ are convex polytopes, then 
    \begin{align*}
        c^{(K)}(A)<\infty \iff \sup_{u\in \mathbb{S}^{n-1}}\rho^{(K_u^*)}(A_u)<\infty.
    \end{align*}
\end{theorem}

The convexification effect of Minkowski summation is classical. It follows from the Shapley--Folkman--Starr theorem \cite{starr1} and the work of Emerson and Greenleaf \cite{emerson1} that the Minkowski averages $A(m):=m^{-1}\sum_{i=1}^{m}A$ of a compact set approach the convex hull of $A$ in Hausdorff distance as $m\rightarrow\infty$. This phenomenon has been studied quantitatively through several measures of nonconvexity. See for example \cite{cassels1,Dyn--Farkhi-2005,fradelizi1}. Recent work in this direction includes results on Hausdorff convexification \cite{meyer3,vanhintum3}, and convexification in Banach spaces and infinite dimensional spaces \cite{artstein2,artstein3}. This paper in contrast studies the question of finite convexification after adding a sufficiently large scaling of a convex set to a nonconvex set. As mentioned above, we recover the nonconvexity index of Schneider \cite{schneider1} when $K=\textup{conv}(A)$. Also of interest is the paper \cite{AFS-2025}, where Artstein-Avidan, Falah, and Slomka study an equivalent formulation of Schneider's index when $A$ is the vertex set of a polytope, and the paper \cite{meyer2} where we quantify the nonconvexity index of Minkowski sums of compact sets.

\subsection{Statement on AI usage} The author used an LLM as a research-assistance tool to help shorten proofs, improve exposition, search for references, and suggest auxiliary lemmas. Notably, the LLM independently found Lemmas \ref{lemma:cover_convex_index_equivalence}, \ref{lemma:one_level_polytope}, and \ref{lemma:polytope_equivalence_simultaneous}, and also Corollary \ref{corollary:projection_covering_condition}. The author is responsible for developing the research problem and mathematical framework. All content in this article was independently reconstructed and verified by the author.

\subsection{Acknowledgments and funding} The author thanks the laboratories LAMA at University Gustave Eiffel, and the Institut Mathématique de Jussieu for their hospitality. The author is supported by the National Science Foundation through the MSPRF program (award number: 2502794).

\section{Special cases}

We begin by studying a few natural cases: strictly convex $K$, lower-dimensional $K$, and a planar convex body $K$.

The following identities will be useful throughout the paper. For $\textbf{u}\in O_j$, note that
\begin{align}\label{eq:usefl_face_identities}
    \textup{conv}(A_{\textbf{u}})=(\textup{conv}(A))_{\textbf{u}},\qquad A_{\textbf{u}}=A\cap (\textup{conv}(A))_{\textbf{u}},
\end{align}
which follow from the definition of iterated exposed sets. 

\subsection{Strictly convex convexifiers}

The set $K\subset\mathbb{R}^n$ is a convex body if it is compact convex with nonempty interior. The convex body $K$ has strictly convex boundary if its boundary does not contain any nondegenerate segments. 

\begin{proposition}\label{proposition:strictly_convex}
    Let $K\subset\mathbb{R}^n$ be a convex body, and let $A\subset\mathbb{R}^n$ be a full-dimensional compact set. If $K$ has strictly convex boundary, then 
    \begin{align*}
        c^{(K)}(A)<\infty \iff \partial \textup{conv}(A)\subset A.
    \end{align*}
\end{proposition}

\begin{proof}
    Suppose that $\partial \textup{conv}(A)\subset A$. Then for each $1\leq j\leq n-1$ and each $\textbf{u}\in O_j$, the set $A_{\textbf{u}}$ is convex. Then $\rho^{(K_{\textbf{u}}^*)}(A_{\textbf{u}})=0$ for all $\textbf{u}$, and therefore 
    \begin{align*}
        \sup_{\textbf{u}\in O_j}\rho^{(K_{\textbf{u}}^*)}(A_{\textbf{u}})<\infty,\qquad 1\leq j\leq n-1.
    \end{align*}
    It follows from Theorem \ref{theorem:main_theorem} that $c^{(K)}(A)<\infty$.

    For the other direction, suppose that $c^{(K)}(A)<\infty$. Then by Theorem \ref{theorem:main_theorem}, $\rho^{(K_{\textbf{u}}^*)}(A_{\textbf{u}})<\infty$ for each $\textbf{u}\in O_j$ and $1\leq j\leq n-1$. By the assumption that $K$ is strictly convex, $K_{\textbf{u}}^*=\{0\}$, and therefore $A_{\textbf{u}}$ is convex for each $\textbf{u}$. If we set $C:=\textup{conv}(A)$, then it follows that $C_{\textbf{u}}=A_{\textbf{u}}$ for each $\textbf{u}$, and 
    \begin{align*}
        \partial C&=\bigcup_{u\in\mathbb{S}^{n-1}}C_{u}\subset A.
    \end{align*}
\end{proof}

\subsection{Lower dimensional convexifiers}\label{subsection:lower_dimensional_convexifiers}

It was mentioned in the introduction that the $G$-covering radius of a compact set $B$ by a convex $G$ may not be finite. The following results show that $\rho^{(G)}(B)<\infty$ if and only if the orthogonal projection of $A$ onto the orthogonal complement of the affine hull of $G$ is a convex set. We denote the affine hull of $G$ by $\textup{aff}(G)$.

\begin{proposition}\label{theorem:hausdorff_distance_finite_condition}
    Let $G\subset \mathbb{R}^n$ be a nonempty compact convex set with $\textup{dim}(G)\geq 1$ that contains $0$ in its relative interior, and let $A\subset\mathbb{R}^n$ be a nonempty compact set. Then
    \begin{align*}
        \rho^{(G)}(A)<\infty \iff (x+\textup{aff}(G))\cap A\neq \varnothing \textup{ for all } x\in \textup{conv}(A).
    \end{align*}
\end{proposition}

\begin{proof}
    Suppose that for every $x\in \textup{conv}(A)$, the affine subspace $x+\textup{aff}(G)$ intersects $A$. Fix $x\in \textup{conv}(A)$ and choose $\alpha \in A\cap (x+\textup{aff}(G))$. Then $x-\alpha \in \textup{aff}(G)$. Since both $x$ and $\alpha$ lie in $\textup{conv}(A)$, 
    \begin{align*}
        \|x-\alpha\|_2\leq \textup{diam}(\textup{conv}(A)).
    \end{align*}
    It follows that if we set
    \begin{align*}
        D:=\textup{diam}(\textup{conv}(A))\qquad \textup{and}\qquad r:=\inf_{x\in \textup{relbd}(G)}\|x\|_2,
    \end{align*}
    then
    \begin{align*}
        x-\alpha\in DB_2^n\cap \textup{aff}(G)\subset \frac{D}{r}G.
    \end{align*}
    Therefore,
    \begin{align*}
        x\in \alpha+\frac{D}{r}G\subset A+\frac{D}{r}G.
    \end{align*}
    Since $x\in\textup{conv}(A)$ was chosen arbitrarily, 
    \begin{align*}
        \textup{conv}(A)\subset A+\frac{D}{r}G,
    \end{align*}
    and we have
    \begin{align*}
        \rho^{(G)}(A)\leq \frac{D}{r}<\infty.
    \end{align*}

    To prove the other direction, suppose that $\rho^{(G)}(A)<\infty$. Then there exists $\lambda>0$ such that $\textup{conv}(A)\subset A+\lambda G$. Fix $x\in \textup{conv}(A)$ and write $x=\alpha+\lambda g$ for some $\alpha\in A$ and $g\in G$. Since $G\subset \textup{aff}(G)$, we have $x-\alpha\in \textup{aff}(G)$. Then $\alpha\in x+\textup{aff}(G)$, and as a consequence $A\cap (x+\textup{aff}(G))\neq \varnothing$.
\end{proof}

\begin{corollary}\label{corollary:projection_covering_condition}
    Let $G\subset\mathbb{R}^n$ be a compact convex set with $\textup{dim}(G)\geq 1$ that contains $0$ in its relative interior. Let $P_{\textup{aff}(G)^{\bot}}$ be the orthogonal projection onto $\textup{aff}(G)^{\bot}$. If $A\subset\mathbb{R}^n$ is a nonempty compact set, then
    \begin{align*}
        \rho^{(G)}(A)<\infty \iff P_{\textup{aff}(G)^{\bot}}(A)\textup{ is convex.}
    \end{align*}
\end{corollary}

\begin{proof}
    Recall from Proposition \ref{theorem:hausdorff_distance_finite_condition} that $\rho^{(G)}(A)<\infty$ if and only if $A\cap (x+\textup{aff}(G))\neq \varnothing$ for each $x\in \textup{conv}(A)$. The latter condition is equivalent to saying that $P_{\textup{aff}(G)^{\bot}}(x)\in P_{\textup{aff}(G)^{\bot}}(A)$ for each $x\in \textup{conv}(A)$. That is, $P_{\textup{aff}(G)^{\bot}}(\textup{conv}(A))\subset P_{\textup{aff}(G)^{\bot}}(A)$. Because linearity of $P_{\textup{aff}(G)^{\bot}}$ gives $P_{\textup{aff}(G)^{\bot}}(\textup{conv}(A))=\textup{conv}(P_{\textup{aff}(G)^{\bot}}(A))$, this condition holds if and only if $P_{\textup{aff}(G)^{\bot}}(A)$ is convex. 
\end{proof}

\subsection{The planar case $n=2$}\label{subsection:uniform_boundedness_essential}

We observe that in the case $n=2$, Theorem \ref{theorem:main_theorem} reduces to a simpler statement that compares lengths of the segments $K_u^*$ with the largest gaps in the sets $A_u$.

\begin{theorem}
    Let $K\subset\mathbb{R}^2$ be a compact convex set with $\textup{dim}(K)\geq 1$, and let $A\subset\mathbb{R}^2$ be a nonempty compact set. Then
    \begin{align*}
        c^{(K)}(A)<\infty \iff \exists \lambda>0\textup{ such that } d(A_u)\leq \lambda |K_u^*| \hspace{2mm}\forall u\in \mathbb{S}^1,
    \end{align*}
    where $|S|$ is the Euclidean length of a segment $S$, and $d(S):=d_H(S,\textup{conv}(S))$ is the Hausdorff distance between a set $S$ and its convex hull.
\end{theorem}

\begin{proof}
    By Theorem \ref{theorem:main_theorem}, we know that 
    \begin{align*}
        c^{(K)}(A)<\infty \iff \sup_{u\in \mathbb{S}^1}\rho^{(K_u^*)}(A_u)<\infty.
    \end{align*}
    If $K_{u}^*$ is a nondegenerate segment, then by definition we have (see also, \cite[Lemma 2.1]{meyer2})
    \begin{align*}
        \rho^{(K_u^*)}(A_u)|K_u^*|=2d(A_u).
    \end{align*}
    In the case that $K_u^*=\{0\}$, then $\rho^{(K_u^*)}(A_u)$ is finite if and only if $A_u$ is convex, in which case $\rho^{(K_u^*)}(A_u)=d(A_u)=0$. The claim follows immediately from these facts.
\end{proof}

It is natural to ask whether the universal boundedness in Theorem \ref{theorem:main_theorem} is necessary, or if it suffices to check that $\rho^{(K_u^*)}(A_u)<\infty$ for each $u\in \mathbb{S}^{n-1}$. The next proposition gives an example in the planar case that shows the uniform boundedness condition must be satisfied.

\begin{proposition}
    There exists a set $K\subset\mathbb{R}^2$ that is compact and convex with $\textup{dim}(K)=2$, and a nonempty compact set $A\subset\mathbb{R}^2$ that satisfy
    \begin{align*}
        \rho^{(K_u^*)}(A_u)<\infty
    \end{align*}
    for each $u\in \mathbb{S}^{1}$, but
    \begin{align*}
        \sup_{u\in \mathbb{S}^{1}}\rho^{(K_u^*)}(A_u)=\infty.
    \end{align*}
\end{proposition}

\begin{proof}
    Let $v_j\in \mathbb{S}^1$ be a sequence of unit vectors that are pairwise nonparallel, set $a_j:=2^{-j-1}$, and set $b_j:=2^{-2j-1}$. Define
    \begin{align*}
        C=\left\{\sum_{j=1}^{\infty}t_ja_jv_j:|t_j|\leq 1\right\},\qquad K=\left\{\sum_{j=1}^{\infty}t_jb_jv_j:|t_j|\leq 1\right\}.
    \end{align*}
    Since $\sum_{j=1}^{\infty}a_j<\infty$ and $\sum_{j=1}^{\infty}b_j<\infty$, the defining series in the sets $C$ and $K$ are absolutely convergent for any $|t_j|\leq 1$, and their tails converge uniformly with respect to the choice of coefficients $t_j\in [-1,1]$. Convexity of $C$ and $K$ follows by adding the series term-wise, and compactness follows by using a bounded subsequence argument together with the uniform tail bound. These sets have nonempty interior since, for example, the directions $v_1$ and $v_2$ are nonparallel. Therefore, $C$ and $K$ are convex bodies in $\mathbb{R}^2$.

    For each $j$, choose $u_j\in \mathbb{S}^1$ such that $\langle u_j,v_j\rangle=0$. Then maximizing the sum
    \begin{align*}
        \sum_{k=1}^{\infty}t_ka_k\langle v_k,u_j\rangle
    \end{align*}
    forces $t_k=\textup{sgn}\langle v_k,u_j\rangle$ when $k\neq j$, but $t_j\in[-1,1]$ is arbitrary. It follows that
    \begin{align*}
        C_{u_j}=x_j+a_j[-v_j,v_j]
    \end{align*}
    for some $x_j\in \mathbb{R}^2$, so that $C_{u_j}$ is a segment of length $2a_j$. Similarly, we find that $K_{u_j}$ is a segment of length $2b_j$. The same is true if we replace $u_j$ with $-u_j$.

    Define
    \begin{align*}
        A:=\partial C\backslash \bigcup_{j=1}^{\infty}(\textup{relint}(C_{u_j})\cup \textup{relint}(C_{-u_j})).
    \end{align*}
    The relative interiors of the sets $C_{u_j}$ and $C_{-u_j}$ are relatively open in $\partial C$. Then $A$ is relatively closed in $\partial C$. Since $\partial C$ is compact, $A$ is also compact.

    By definition of $A$, each set $A_{u_j}$ consists of the two endpoints of $C_{u_j}$ that are separated by a distance of $2a_j$. The set $K_{u_j}^{*}$ is an interval centered at $0$ and of length $2b_j$. We calculate
    \begin{align*}
        \rho^{(K_{u_j}^*)}(A_{u_j})=\frac{a_j}{b_j}=2^j.
    \end{align*}
    Now, choose $u\in\mathbb{S}^1$ such that $u\neq \pm u_j$ for any $j$. Similar to the above calculation, we find that $C_{u}$ is a point, and therefore $A_u$ is a point. In this case, $\rho^{(K_u^*)}(A_u)=0$. That is, $\rho^{(K_u^*)}(A_u)<\infty$ for each $u\in\mathbb{S}^1$, but considering the sequence $u_j$, we have
    \begin{align*}
        \sup_{u\in\mathbb{S}^1}\rho^{(K_u^*)}(A_u)=\infty.
    \end{align*}
\end{proof}

\section{Proof of Theorem \ref{theorem:finite_condition_universal_translate}}

\begin{lemma}\label{lemma:characterizing_sumset_lemma}
    If $K,V\subset\mathbb{R}^n$ are nonempty compact convex sets such that $0\in \textup{relint}(K)$, then
    \begin{equation}\label{eq:characterization_union_formula}
        V+K=V\cup \bigcup_{u\in \mathbb{S}^{n-1}}\textup{conv}(V_u,(V+K)_u).
    \end{equation}
\end{lemma}

\begin{proof}
    Let $x\in (V+K)\backslash V$. Set
    \begin{align*}
        \tau:=\min\{t\in [0,1]:x\in V+tK\},
    \end{align*}
    which is well-defined by compactness of $V$ and $K$. Since $x\notin V$, we must have $\tau>0$. Choose a sequence $t_k\rightarrow \tau$ with $t_k<\tau$. Then $x\notin V+t_k K$, and there exists a hyperplane that strictly separates $x$ from $V+t_k K$. That is, there exists $u_k\in \mathbb{S}^{n-1}$ for which 
    \begin{align*}
        \langle x,u_k \rangle>h_{V}(u_k)+t_kh_K(u_k).
    \end{align*}
    Since the directions $u_k$ belong to the compact $\mathbb{S}^{n-1}$, there exists $u\in \mathbb{S}^{n-1}$ such that up to subsequences, $u_k\rightarrow u$. By continuity of the support function, we have
    \begin{align*}
        \langle x,u\rangle \geq h_V(u)+\tau h_K(u).
    \end{align*}
    Since $x\in V+\tau K$, the above inequality also holds in the other direction, and is therefore an equality. It follows that 
    \begin{align*}
        x\in (V+\tau K)_u=V_u+\tau K_u.
    \end{align*}
    Write $x=v+\tau k$, with $v\in V_u$ and $k\in K_u$. Then $v+k\in (K+V)_u$, and 
    \begin{align*}
        x=(1-\tau)v+\tau (v+k)\in \textup{conv}(V_u,(V+K)_u).
    \end{align*}
    We have shown that 
    \begin{align*}
        V+K\subset V\cup \bigcup_{u\in \mathbb{S}^{n-1}}\textup{conv}(V_u,(V+K)_u).
    \end{align*}
    The other direction follows immediately from the fact that $0\in K$, so we have established the desired relation. 
\end{proof}

\begin{lemma}\label{lemma:translate_face_containment}
    Let $G\subset\mathbb{R}^n$ be a nonempty compact convex set such that $g(G)=0$. If $u\in \mathbb{S}^{n-1}$ and $p\in G_u$, then 
    \begin{align*}
        \frac{1}{m+1}(G_u-p)\subset  G,
    \end{align*}
    where $m:=\textup{dim}(G)$.
\end{lemma}

\begin{proof}
    By applying \cite[Lemma 1.11.4]{BKFR-2026} in the $m$-dimensional subspace $\textup{aff}(G)$, we have that $-G\subset m G$. If $p\in G$, then
    \begin{align*}
        G-p\subset G+(-G)\subset (1+m)G.
    \end{align*}
    This proves what we need, since if $p\in G_u$, then $p\in G$ and $G_u\subset G$.
\end{proof}

\begin{lemma}\label{lemma:finite_convex_faces_condition}
    Let $K\subset\mathbb{R}^n$ be a compact convex set with $\textup{dim}(K)\geq 1$, and let $A\subset\mathbb{R}^n$ be a nonempty compact set. Set
    \begin{align*}
        \Lambda^{(K)}(A):=\sup_{u\in \mathbb{S}^{n-1}}c^{(K_u)}(A_u).
    \end{align*}
    Then
    \begin{align*}
        c^{(K)}(A)<\infty \iff \Lambda^{(K)}(A)<\infty.
    \end{align*}
    Moreover, if $\Lambda^{(K)}(A)<\infty$, then $\rho^{(K^*)}(A)<\infty$ and
    \begin{align*}
        c^{(K)}(A)\leq \max\{\rho^{(K^*)}(A),(1+m)\Lambda^{(K)}(A)\},
    \end{align*}
    where $m:=\textup{dim}(K)$.
\end{lemma}

\begin{proof}
    Suppose that $c^{(K)}(A)<\infty$. Choose $r>0$ such that $A+rK$ is convex. For every $u\in \mathbb{S}^{n-1}$, we have $(A+rK)_u=A_u+rK_u$. Since the exposed faces of a convex set are convex, the set $A_u+rK_u$ is convex. It follows that
    \begin{align*}
        c^{(K_u)}(A_u)\leq r
    \end{align*}
    for each $u\in \mathbb{S}^{n-1}$, and therefore $\Lambda ^{(K)}(A)\leq r<\infty$.

    To prove the other direction, assume that $\Lambda:=\Lambda^{(K)}(A)<\infty$. Both $c^{(K)}(A)$ and $\Lambda^{(K)}(A)$ are invariant under translation of $K$, so we will assume that the relative centroid of $K$ is the origin. That is, $K=K^*$. Set $C:=\textup{conv}(A)$ and $L:=\textup{aff}(K)$. Because $0\in K$, $L$ is a linear subspace.

    We will show that $R:=\rho^{(K)}(A)<\infty$. Choose $\mu>\Lambda$. We claim that for every $x\in \partial C$, there exists $a\in A$ such that $x-a\in L$. To prove this, choose $u\in \mathbb{S}^{n-1}$ such that $x\in C_u$. Since
    \begin{align*}
        c^{(K_u)}(A_u)\leq \Lambda<\mu,
    \end{align*}
    the set $A_u+\mu K_u$ is convex. We have
    \begin{align*}
        A_u+\mu K_u&=\textup{conv}(A_u+\mu K_u)\\
        &=\textup{conv}(A_u)+\mu K_u\\
        &=C_u+\mu K_u.
    \end{align*}
    Choose $p\in K_u$. Since $x\in C_u$, the above relation implies that $x+\mu p=a+\mu q$, for some $a\in A_u$ and $q\in K_u$. Then $x-a=\mu(q-p)\in L$, verifying the claim. Next, we will show that the same claim is true if we choose arbitrary $x\in C$. If $\textup{int}(C)=\varnothing$, then $C=\partial C$, and there is nothing to prove. Suppose that $\textup{int}(C)$ is nonempty, and choose $x\in \textup{int}(C)$. Since $\textup{dim}(L)\geq 1$, there exists a nonzero $v\in L$. Since $C$ is a convex body, the line $x+\mathbb{R}v$ intersects $\partial C$ at some point $y$. By the previous claim, there exists $a\in A$ such that $y-a\in L$. Since $x-y\in L$, we have 
    \begin{align*}
        x-a=(x-y)+(y-a)\in L.
    \end{align*}
    That is, for every $x\in C$, the set $A\cap (x+L)$ is nonempty. It follows from Proposition \ref{theorem:hausdorff_distance_finite_condition} that $R<\infty$.

    Next, we will verify that $c^{(K)}(A)<\infty$. For $\gamma:=(m+1)^{-1}$, choose $\mu>\max\{\gamma R,\Lambda\}$, and set $D:=\gamma^{-1}\mu$. Then $D>R$. Fix $u\in \mathbb{S}^{n-1}$ and choose $p=g(K_u)\in K_u$. By Lemma \ref{lemma:translate_face_containment}, both $\gamma(K_u-p)$ and $K_u$ are contained in $K$. Since $K$ is convex,
    \begin{align*}
        \textup{conv}(\gamma(K_u-p),K_u)\subset K.
    \end{align*}
    Consequently,
    \begin{align*}
        A+DK&\supset A_u+D\hspace{0.5mm}\textup{conv}(\gamma(K_u-p),K_u)\\
        &=\bigcup_{t\in [0,1]}\left[A_u+D(t\gamma+(1-t))K_u-\mu tp\right].
    \end{align*}
    For $t\in [0,1]$, set
    \begin{align*}
        r_t:=D(t\gamma +(1-t))=\mu t+D(1-t).
    \end{align*}
    Since $\gamma\leq 1$, we have $D\geq \mu$, and therefore $r_t\geq \mu>\Lambda$. Then $A_u+r_t K_u$ is convex, and just as above
    \begin{align*}
        A_u+r_t K_u=C_u+r_t K_u.
    \end{align*}
    By substitution, we have
    \begin{align*}
        A+DK&\supset \bigcup_{t\in [0,1]}\left[C_u+\mu t(K_u-p)+(1-t)DK_u\right]\\
        &\supset \bigcup_{t\in [0,1]}\left[C_u+(1-t)DK_u\right],
    \end{align*}
    where we used that $0\in K_u-p$. Since $C_u$ is convex,
    \begin{align*}
        C_u+(1-t)DK_u=tC_u+(1-t)(C_u+DK_u).
    \end{align*}
    Also,
    \begin{align*}
        (C+DK)_u=C_u+DK_u.
    \end{align*}
    Therefore,
    \begin{align*}
        A+DK\supset \textup{conv}(C_u,(C+DK)_u).
    \end{align*}
    Since $u$ was arbitrary,
    \begin{align*}
        A+DK\supset \bigcup_{u\in \mathbb{S}^{n-1}}\textup{conv}(C_u,(C+DK)_u).
    \end{align*}
    Because $D>R$, we have $C\subset A+DK$. Now, use Lemma \ref{lemma:characterizing_sumset_lemma} to get
    \begin{align*}
        A+DK\supset C\cup \bigcup_{u\in \mathbb{S}^{n-1}}\textup{conv}(C_u,(C+DK)_u)=C+DK.
    \end{align*}
    The inclusion $A+DK\subset C+DK$ is immediate because $A\subset C$. Therefore, $A+DK=C+DK$, and the right side is convex. But this is true for arbitrary $D>\max\{R,\gamma^{-1}\Lambda\}$, so we obtain the bound
    \begin{align*}
        c^{(K)}(A)\leq \max\{R,\gamma^{-1}\Lambda\}<\infty,
    \end{align*}
    as claimed.
\end{proof}

\begin{lemma}\label{lemma:cover_convex_index_equivalence}
    Let $G\subset\mathbb{R}^n$ be nonempty, compact, and convex and let $B\subset\mathbb{R}^n$ be nonempty and compact. Define 
    \begin{align*}
        \Delta^{(G)}(B):=\sup_{p\in \textup{relint}(G)}\rho^{(G-p)}(B).
    \end{align*}
    Then as extended real numbers, $c^{(G)}(B)=\Delta^{(G)}(B)$.
\end{lemma}

\begin{proof}
    Suppose that $c^{(G)}(B)<\infty$ and let $t>c^{(G)}(B)$. Then $B+tG$ is convex. Fix $p\in \textup{relint}(G)$. Since translation does not affect convexity, $B+t(G-p)$ is convex. Therefore,
    \begin{align*}
        B+t(G-p)=\textup{conv}(B)+t(G-p).
    \end{align*}
    Since $0\in G-p$, we have 
    \begin{align*}
        \textup{conv}(B)\subset \textup{conv}(B)+t(G-p)=B+t(G-p).
    \end{align*}
    Then $\rho^{(G-p)}(B)\leq t$. Since this holds for every $p\in \textup{relint}(G)$, we have $\Delta^{(G)}(B)\leq t$. Taking $t\rightarrow c^{(G)}(B)$, we get $\Delta^{(G)}(B)\leq c^{(G)}(B)$.

    Next, suppose that $\Delta^{(G)}(B)<\infty$. Choose $t>\Delta^{(G)}(B)$. Then for every $p\in \textup{relint}(G)$, we have $\rho^{(G-p)}(B)<t$. Therefore, 
    \begin{align*}
        \textup{conv}(B)\subset B+t(G-p).
    \end{align*}
    Rearranging the above containment, we get
    \begin{align*}
        \textup{conv}(B)+tp\subset B+tG.
    \end{align*}
    But this holds for every $p\in \textup{relint}(G)$, so taking the union of the above containment over all $p\in \textup{relint}(G)$, we get
    \begin{align*}
        \textup{conv}(B)+t\hspace{0.5mm}\textup{relint}(G)\subset B+tG.
    \end{align*}
    The set $B+tG$ is closed, and $\textup{cl}(\textup{relint}(G))=G$, so taking the closure, we get
    \begin{align*}
        \textup{conv}(B)+tG\subset B+tG.
    \end{align*}
    The inclusion holds in the other direction, since $B\subset \textup{conv}(B)$. We have $\textup{conv}(B)+tG=B+tG$, which implies $c^{(G)}(B)\leq t$. Taking $t\rightarrow\Delta^{(G)}(B)$, we have $c^{(G)}(B)\leq \Delta^{(G)}(B)$.
\end{proof}

\begin{proof}[Proof of Theorem \ref{theorem:finite_condition_universal_translate}]
    We claim that
    \begin{align*}
        \Lambda^{(K)}(A)<\infty \iff \sup_{u\in \mathbb{S}^{n-1}}\Delta^{(K_u)}(A_u)<\infty.
    \end{align*}
    The theorem will then follow from Lemma \ref{lemma:finite_convex_faces_condition}. By Lemma \ref{lemma:cover_convex_index_equivalence},
    \begin{align*}
        \Lambda^{(K)}(A)&=\sup_{u\in \mathbb{S}^{n-1}}c^{(K_u)}(A_u)\\
        &=\sup_{u\in \mathbb{S}^{n-1}}\Delta^{(K_u)}(A_u).
    \end{align*}
    The claim follows immediately.
\end{proof}

\section{Proof of Theorem \ref{theorem:main_theorem}}

\begin{lemma}
    Let $B\subset\mathbb{R}^n$ be a nonempty compact set, let $G\subset\mathbb{R}^n$ be a nonempty compact convex set, and let $\tilde{G}$ be any translate of $G$ that contains $0$ in its relative interior. Then
    \begin{align}\label{eq:covering_bound_schneider}
        \rho^{(\tilde{G})}(B)\leq c^{(G)}(B),
    \end{align}
    and we also have
    \begin{align}\label{eq:interative_bound_schneider}
        \sup_{v\in \mathbb{S}^{n-1}}c^{(G_v)}(B_v)\leq c^{(G)}(B).
    \end{align}
\end{lemma}

\begin{proof}
    If $c^{(G)}(B)=\infty$, then both statements are immediate. Assume that $c^{(G)}(B)<\infty$. To prove the first bound, let $t>c^{(G)}(B)$, so that $B+tG$ is convex. Translating $G$ does not affect the convexity of the sum, so $B+t\tilde{G}$ is convex. Since $0\in \tilde{G}$, we have $B\subset B+t\tilde{G}$, which by convexity of the set on the right implies that $\textup{conv}(B)\subset B+t\tilde{G}$. Then $\rho^{(\tilde{G})}(B)\leq t$. Letting $t\rightarrow c^{(G)}(B)$ verifies \eqref{eq:covering_bound_schneider}.

    To prove the second bound, with $t$ chosen as above, we have that $B+tG$ is convex. Then every exposed support set is convex. That is, for each $v\in \mathbb{S}^{n-1}$, the set
    \begin{align*}
        (B+tG)_v=B_v+tG_v
    \end{align*}
    is convex. It follows that for each $v$, we have 
    \begin{align*}
        c^{(G_v)}(B_v)\leq c^{(G)}(B),
    \end{align*}
    which proves \eqref{eq:interative_bound_schneider} after taking the supremum.
\end{proof}

\begin{lemma}\label{lemma:iterative_bounded_faces}
    Suppose that $\Lambda^{(K)}(A)\leq M$. For $1\leq j\leq n-1$ and $\textbf{u}\in O_j$, we have
        \begin{align}\label{eq:uniform_iterative_bound}
            c^{(K_{\textbf{u}})}(A_{\textbf{u}})\leq M.
        \end{align}
    As a consequence, if $\tilde{K}_{\textbf{u}}$ is a translate of $K_{\textbf{u}}$ containing $0$ in its relative interior, then 
        \begin{align}
            \rho^{(\tilde{K}_{\textbf{u}})}(A_{\textbf{u}})\leq M.
        \end{align}
\end{lemma}

\begin{proof}
    We will prove the lemma by induction. For the case $j=1$, we have $\textbf{u}=(u_1)$, and the statement $c^{(K_{u_1})}(A_{u_1})\leq M$ follows immediately from the fact that $\Lambda ^{(K)}(A)\leq M$.
    
    For the induction step, assume that $c^{(K_{\textbf{u}})}(A_{\textbf{u}})\leq M$ for some $\textbf{u}\in O_j$, with $1\leq j<n-1$. Apply bound \eqref{eq:interative_bound_schneider} with $B:=A_{\textbf{u}}$ and $G:=K_{\textbf{u}}$ to get
    \begin{align*}
        \sup_{v}c^{(K_{(\textbf{u},v)})}(A_{(\textbf{u},v)})\leq c^{(K_{\textbf{u}})}(A_{\textbf{u}})\leq M.
    \end{align*}
    This implies that for every admissible $(\textbf{u},v)\in O_{j+1}$, we have
    \begin{align*}
        c^{(K_{(\textbf{u},v)})}(A_{(\textbf{u},v)})\leq M,
    \end{align*}
    showing that if \eqref{eq:uniform_iterative_bound} holds for $j$, then it also holds for $j+1$. By induction, we have $c^{(K_{\textbf{u}})}(A_{\textbf{u}})\leq M$ for any $\textbf{u}\in O_j$ with $1\leq j\leq n-1$.

    For the last part of the lemma, it follows from \eqref{eq:covering_bound_schneider} with $G:=K_{\textbf{u}}$ and $B:=A_{\textbf{u}}$ that
    \begin{align*}
        \rho^{(\tilde{K}_{\textbf{u}})}(A_{\textbf{u}})\leq c^{(K_{\textbf{u}})}(A_{\textbf{u}})\leq M
    \end{align*}
    for every translate $\tilde{K}_{\textbf{u}}$ containing $0$ in its relative interior, which is the desired bound.
\end{proof}

\begin{proof}[Proof of Theorem \ref{theorem:main_theorem}]
    Let
    \begin{align*}
        R_j:=\sup_{\textbf{u}\in O_j}\rho^{(K_{\textbf{u}}^*)}(A_{\textbf{u}}),\qquad C_j:=\sup_{\textbf{u}\in O_j}c^{(K_{\textbf{u}})}(A_{\textbf{u}}).
    \end{align*}
    Assume that the covering condition on the right side of \eqref{eq:covering_radius_finite_condition} holds, so that $R_j<\infty$ for every $1\leq j\leq n-1$. We will show that $C_1<\infty$ and use Lemma \ref{lemma:finite_convex_faces_condition} to conclude that $c^{(K)}(A)<\infty$.
    
    We will use backward induction, starting at $j=n-1$. For $\textbf{u}\in O_{n-1}$, the sets $K_{\textbf{u}}$ and $A_{\textbf{u}}$ have affine dimension at most $1$. If $K_{\textbf{u}}$ is a point, then $K_{\textbf{u}}^*=\{0\}$ and the condition $\rho^{(\{0\})}(A_{\textbf{u}})<\infty$ is equivalent to $A_{\textbf{u}}=\textup{conv}(A_{\textbf{u}})$. It follows that $A_{\textbf{u}}$ is convex and $c^{(K_{\textbf{u}})}(A_{\textbf{u}})=\rho^{(K_{\textbf{u}}^*)}(A_{\textbf{u}})=0$. If $K_{\textbf{u}}$ is a nondegenerate segment, then $c^{(K_{\textbf{u}})}(A_{\textbf{u}})=\rho^{(K_{\textbf{u}}^*)}(A_{\textbf{u}})$. Consequently, $C_{n-1}=R_{n-1}<\infty$.

    Next, assume that $C_{j+1}<\infty$ for some $1\leq j\leq n-2$. Fix $\textbf{u}\in O_{j}$. If $K_{\textbf{u}}$ is a point, then just as above, finiteness of $R_j$ forces $A_{\textbf{u}}$ to be convex, and we have $c^{(K_{\textbf{u}})}(A_{\textbf{u}})=0$. Suppose that $\textup{dim}(K_{\textbf{u}})\geq 1$. Both $K_{\textbf{u}}$ and $A_{\textbf{u}}$ lie in translates of the subspace $H_{\textbf{u}}:=\textup{span}\{u_1,\dots,u_j\}^{\bot}$. By invariance under translation, we may translate $K_{\textbf{u}}$ and $A_{\textbf{u}}$ so that they lie in $H_{\textbf{u}}$. If $v\in \mathbb{S}^{n-1}\cap H_{\textbf{u}}$, then $(\textbf{u},v)\in O_{j+1}$, and therefore the quantity $\Lambda^{(K_{\textbf{u}})}(A_{\textbf{u}})$ computed in $H_{\textbf{u}}$ is at most $C_{j+1}$. Now, apply Lemma \ref{lemma:finite_convex_faces_condition} within $H_{\textbf{u}}$ to get
    \begin{align*}
        c^{(K_{\textbf{u}})}(A_{\textbf{u}})\leq \max\left\{\rho^{(K_{\textbf{u}}^*)}(A_{\textbf{u}}),(1+\textup{dim}(K_{\textbf{u}}))C_{j+1}\right\}\leq \max\{R_{j},(1+n-j)C_{j+1}\}.
    \end{align*}
    Taking the supremum over all $\textbf{u}\in O_j$ gives $C_j<\infty$. By induction, we conclude that $C_1<\infty$, and consequently $c^{(K)}(A)<\infty$.

    To prove the other direction, suppose that $c^{(K)}(A)<\infty$. Then $\Lambda^{(K)}(A)\leq c^{(K)}(A)$. Apply Lemma \ref{lemma:iterative_bounded_faces} with $M:=c^{(K)}(A)$ to conclude that $R_j\leq M$ for each $1\leq j\leq n-1$.
\end{proof}

\section{Proof of Theorem \ref{theorem:polygone_equivalence}}\label{subsection:polyhedral_convexifiers}

\begin{lemma}\label{lemma:one_level_polytope}
    Let $P\subset\mathbb{R}^n$ be a nonempty convex polytope, and let $\textbf{u}\in O_j$. There exists $\varepsilon_0>0$ such that if $0<\varepsilon<\varepsilon_0$, then 
    \begin{align*}
        P_{w_{\varepsilon}}=P_{\textbf{u}},
    \end{align*}
    where
    \begin{align*}
        w_{\varepsilon}=u_1+\varepsilon u_2+\dots+\varepsilon^{j-1}u_j.
    \end{align*}
\end{lemma}

\begin{proof}
    Let $V_0$ denote the vertex set of $P$. For $1\leq r\leq j$, define 
    \begin{align*}
        V_r:=\left\{x\in V_{r-1}:\langle x,u_r\rangle=\max_{y\in V_{r-1}}\langle y,u_r\rangle\right\}.
    \end{align*}
    In fact, the set $V_r$ is the vertex set of the exposed face $P_{(u_1,\dots,u_r)}$, so we have
    \begin{align*}
        P_{(u_1,\dots,u_r)}=\textup{conv}(V_r).
    \end{align*}
    Fix $x\in V_j$ and $y\in V_0\backslash V_j$. Let $r$ be the smallest index such that $y$ does not belong to $V_r$. Since the sets $V_i$ are nested in the decreasing sense, we have that $x$ belongs to all the $V_i$, and so
    \begin{align*}
        \langle x,u_i\rangle=\langle y,u_i\rangle\qquad i<r,
    \end{align*}
    but
    \begin{align*}
        \langle x,u_r\rangle >\langle y,u_r\rangle.
    \end{align*}
    Next, compute
    \begin{align*}
        \langle x-y,w_{\varepsilon}\rangle=\varepsilon^{r-1}\langle x-y,u_r\rangle+\sum_{i=r+1}^{j}\varepsilon^{i-1}\langle x-y,u_i\rangle.
    \end{align*}
    Since the leading term is positive, the above quantity is positive for $\varepsilon>0$ sufficiently small. Since $P$ has finitely many vertices, there exists $\varepsilon_0$ so small that the above quantity is positive for any choice of $x\in V_j$ and $y\in V_0\backslash V_j$. Now, take $x,x^{\prime}\in V_j$. Then $\langle x,u_i\rangle=\langle x^{\prime},u_i\rangle$ for every $1\leq i\leq j$, and therefore $\langle x,w_{\varepsilon}\rangle=\langle x^{\prime},w_{\varepsilon}\rangle$. It follows that the vertices $x\in V_0$ that maximize $\langle x,w_{\varepsilon}\rangle$ are exactly the vertices $V_j$. Therefore, since the face of a convex polytope is the convex hull of its vertices, we have
    \begin{align*}
        P_{w_{\varepsilon}}=\textup{conv}(V_j)=P_{\textbf{u}}.
    \end{align*}
\end{proof}

\begin{lemma}\label{lemma:polytope_equivalence_simultaneous}
    Let $K\subset\mathbb{R}^n$ be a nonempty compact convex set, and let $A\subset\mathbb{R}^n$ be a nonempty compact set. If both $K$ and $\textup{conv}(A)$ are polytopes that are not necessarily full-dimensional, then for every $\textbf{u}\in O_j$ there exists $w\in \mathbb{S}^{n-1}$ such that 
    \begin{align*}
        K_w=K_{\textbf{u}}\qquad \textup{and}\qquad A_w=A_{\textbf{u}}.
    \end{align*}
\end{lemma}

\begin{proof}
    Set $C:=\textup{conv}(A)$. By Lemma \ref{lemma:one_level_polytope}, there exist positive $\varepsilon_K$ and $\varepsilon_C$ such that 
    \begin{align*}
        K_{w_{\varepsilon}}=K_{\textbf{u}}
    \end{align*}
    for $0<\varepsilon<\varepsilon_K$, and
    \begin{align*}
        C_{w_{\varepsilon}}=C_{\textbf{u}}
    \end{align*}
    for $0<\varepsilon<\varepsilon_C$. Choose $0<\varepsilon <\min \{\varepsilon_K,\varepsilon_C\}$, and set $w:=w_{\varepsilon}/\|w_{\varepsilon}\|_2$. Since normalization does not change the exposed faces, we have
    \begin{align*}
        K_{w}=K_{\textbf{u}},\qquad C_w=C_{\textbf{u}}.
    \end{align*}
    Finally, by \eqref{eq:usefl_face_identities} we have
    \begin{align*}
        A_{w}=A\cap C_w=A\cap C_{\textbf{u}}=A_{\textbf{u}},
    \end{align*}
    which proves what we need.
\end{proof}

\begin{proof}[Proof of Theorem \ref{theorem:polygone_equivalence}]
    By Theorem \ref{theorem:main_theorem}, it suffices to show that 
    \begin{align*}
        \sup_{\textbf{u}\in O_j}\rho^{(K_{\textbf{u}
        }^*)}(A_{\textbf{u}})\leq \sup_{u\in \mathbb{S}^{n-1}}\rho^{(K_u^*)}(A_u)
    \end{align*}
    for each $1\leq j\leq n-1$. To see this, fix $\textbf{u}\in O_j$. By Lemma \ref{lemma:polytope_equivalence_simultaneous}, there exists $w\in \mathbb{S}^{n-1}$ such that 
    \begin{align*}
        K_{\textbf{u}}=K_w,\qquad A_{\textbf{u}}=A_w.
    \end{align*}
    Identical convex sets have identical relative centroids, so we have
    \begin{align*}
        \rho^{(K_{\textbf{u}}^*)}(A_{\textbf{u}})=\rho^{(K_{w}^*)}(A_w)\leq \sup_{u\in \mathbb{S}^{n-1}}\rho^{(K_u^*)}(A_u).
    \end{align*}
    The result follows by taking the supremum over $\textbf{u}\in O_j$.
\end{proof}

For the next result, note that convex bodies $K,L\subset\mathbb{R}^n$ are called strongly isomorphic if for each $u\in \mathbb{S}^{n-1}$, we have 
\begin{align*}
    \textup{dim}(K_u)=\textup{dim}(L_u).
\end{align*}
That is, corresponding exposed faces have the same dimension.

\begin{corollary}\label{corollary:strongly_isomorphic_polytopes}
    Let $K\subset\mathbb{R}^n$ be a convex body, and let $A\subset\mathbb{R}^n$ be a full-dimensional compact set. If $K$ and $\textup{conv}(A)$ are strongly isomorphic convex polytopes, then $c^{(K)}(A)<\infty$.
\end{corollary}

\begin{proof}[Proof of Corollary \ref{corollary:strongly_isomorphic_polytopes}]
Set $C:=\textup{conv}(A)$ and let $u\in \mathbb{S}^{n-1}$. If $K_u$ is a point, then the strong isomorphism of $K$ and $C$ implies that $C_u$ is a point. Then $C_u=A_u$ and $\rho^{(K_u^*)}(A_u)=0$. Suppose that $\textup{dim}(K_u)\geq 1$. Since $K$ and $C$ are strongly isomorphic, the affine hulls of $K_u$ and $C_u$ are translates of each other; see the proof of Corollary 2.4.12 in Schneider \cite[Page 110]{schneider-2014}. Then
\begin{align*}
    \textup{span}(K_u-K_u)=\textup{span}(C_u-C_u).
\end{align*}
Since $0\in K_u^*$, we have 
\begin{align*}
    \textup{aff}(K_u^*)=\textup{span}(K_u-K_u)=\textup{span}(C_u-C_u).
\end{align*}
Let $x\in \textup{conv}(A_u)$. By \eqref{eq:usefl_face_identities}, we have $x\in C_u$, and therefore
\begin{align*}
    x+\textup{aff}(K_u^*)=\textup{aff}(C_u).
\end{align*}
Since $A_u$ is a nonempty subset of $\textup{aff}(C_u)$, we see that $x+\textup{aff}(K_u^*)$ intersects $A_u$. By Proposition \ref{theorem:hausdorff_distance_finite_condition}, we have that $\rho^{(K_u^*)}(A_u)<\infty$ for every $u\in \mathbb{S}^{n-1}$. Since $K$ and $C$ are polytopes, there are finitely many distinct pairs $(K_u,A_u)$, so we conclude that 
\begin{align*}
    \sup_{u\in \mathbb{S}^{n-1}}\rho^{(K_u^*)}(A_u)<\infty.
\end{align*}
The result follows from Theorem \ref{theorem:polygone_equivalence}.
\end{proof}

\bibliography{bibliography}

@article {artstein2,
    AUTHOR = {Artstein, Zvi},
     TITLE = {Convexification estimates for {M}inkowski averages in infinite
              dimensions},
   JOURNAL = {Pure Appl. Funct. Anal.},
  FJOURNAL = {Pure and Applied Functional Analysis},
    VOLUME = {6},
      YEAR = {2021},
    NUMBER = {4},
     PAGES = {709--718},
      ISSN = {2189-3756,2189-3764},
   MRCLASS = {52A05 (41A46 46A55 46B50 52A39)},
  MRNUMBER = {4302897},
}

@article {artstein3,
    AUTHOR = {Artstein, Zvi and Kadets, Vladimir},
     TITLE = {B-convexity, convexification of {M}inkowski averages in a
              {B}anach space, and {SLLN} for random sets},
   JOURNAL = {J. Convex Anal.},
  FJOURNAL = {Journal of Convex Analysis},
    VOLUME = {32},
      YEAR = {2025},
    NUMBER = {1},
     PAGES = {61--70},
      ISSN = {0944-6532,2363-6394},
   MRCLASS = {46B20 (28B20)},
  MRNUMBER = {4950448},
}

@article {AFS-2025,
    AUTHOR = {Artstein-Avidan, Shiri and Falah, Tomer and Slomka, Boaz A.},
     TITLE = {Vertex generated polytopes},
   JOURNAL = {Pure Appl. Funct. Anal.},
  FJOURNAL = {Pure and Applied Functional Analysis},
    VOLUME = {10},
      YEAR = {2025},
    NUMBER = {5},
     PAGES = {1171--1192},
      ISSN = {2189-3756,2189-3764},
   MRCLASS = {52A40 (52B12 52C17)},
  MRNUMBER = {5000520},
}

@article {schneider1,
    AUTHOR = {Schneider, Rolf},
     TITLE = {A measure of convexity for compact sets},
   JOURNAL = {Pacific J. Math.},
  FJOURNAL = {Pacific Journal of Mathematics},
    VOLUME = {58},
      YEAR = {1975},
    NUMBER = {2},
     PAGES = {617--625},
      ISSN = {0030-8730,1945-5844},
   MRCLASS = {52A20},
  MRNUMBER = {380628},
MRREVIEWER = {H.\ W.\ Guggenheimer},
       URL = {http://projecteuclid.org/euclid.pjm/1102905693},
}

@article {cassels1,
    AUTHOR = {Cassels, J. W. S.},
     TITLE = {Measures of the non-convexity of sets and the
              {S}hapley-{F}olkman-{S}tarr theorem},
   JOURNAL = {Math. Proc. Cambridge Philos. Soc.},
  FJOURNAL = {Mathematical Proceedings of the Cambridge Philosophical
              Society},
    VOLUME = {78},
      YEAR = {1975},
    NUMBER = {3},
     PAGES = {433--436},
      ISSN = {0305-0041,1469-8064},
   MRCLASS = {52A05 (90A99)},
  MRNUMBER = {385711},
MRREVIEWER = {Jacques\ Dubois},
       DOI = {10.1017/S0305004100051884},
       URL = {https://doi.org/10.1017/S0305004100051884},
}

@article {Dyn--Farkhi-2005,
    AUTHOR = {Dyn, Nira and Farkhi, Elza},
     TITLE = {Set-valued approximations with {M}inkowski
              averages---convergence and convexification rates},
   JOURNAL = {Numer. Funct. Anal. Optim.},
  FJOURNAL = {Numerical Functional Analysis and Optimization. An
              International Journal},
    VOLUME = {25},
      YEAR = {2004},
    NUMBER = {3-4},
     PAGES = {363--377},
      ISSN = {0163-0563,1532-2467},
   MRCLASS = {41A30},
  MRNUMBER = {2072074},
       DOI = {10.1081/NFA-120039682},
       URL = {https://doi.org/10.1081/NFA-120039682},
}

@article {emerson1,
    AUTHOR = {Emerson, W. R. and Greenleaf, F. P.},
     TITLE = {Asymptotic behavior of products {$C\sp{p}=C+\cdots +C$} in
              locally compact abelian groups},
   JOURNAL = {Trans. Amer. Math. Soc.},
  FJOURNAL = {Transactions of the American Mathematical Society},
    VOLUME = {145},
      YEAR = {1969},
     PAGES = {171--204},
      ISSN = {0002-9947,1088-6850},
   MRCLASS = {22.20},
  MRNUMBER = {249535},
MRREVIEWER = {L.\ Corwin},
       DOI = {10.2307/1995065},
       URL = {https://doi.org/10.2307/1995065},
}

@book {BKFR-2026,
    AUTHOR = {B\"or\"oczky, K\'aroly J. and Figalli, Alessio and Ramos,
              Jo\~ao P. G.},
     TITLE = {Isoperimetric inequalities, {B}runn-{M}inkowski theory and
              {M}inkowski-type {M}onge-{A}mp\`ere equations on the sphere},
    SERIES = {Zurich Lectures in Advanced Mathematics},
 PUBLISHER = {EMS Press, Berlin},
      YEAR = {[2026] \copyright 2026},
     PAGES = {xv+523},
      ISBN = {978-3-98547-104-1; 978-3-98547-604-6},
   MRCLASS = {52A40 (26D10 26D15 35J96 49Q10 49Q15 52A39)},
  MRNUMBER = {5095086},
}

@article {fradelizi1,
    AUTHOR = {Fradelizi, Matthieu and Madiman, Mokshay and Marsiglietti,
              Arnaud and Zvavitch, Artem},
     TITLE = {The convexification effect of {M}inkowski summation},
   JOURNAL = {EMS Surv. Math. Sci.},
  FJOURNAL = {EMS Surveys in Mathematical Sciences},
    VOLUME = {5},
      YEAR = {2018},
    NUMBER = {1-2},
     PAGES = {1--64},
      ISSN = {2308-2151,2308-216X},
   MRCLASS = {60E15 (11B13 60F15 94A17)},
  MRNUMBER = {3880220},
       DOI = {10.4171/EMSS/26},
       URL = {https://doi.org/10.4171/EMSS/26},
}

@article {meyer2,
    AUTHOR = {Meyer, Mark},
     TITLE = {Measuring the convexity of compact sumsets with the
              {S}chneider non-convexity index},
   JOURNAL = {J. Inequal. Appl.},
  FJOURNAL = {Journal of Inequalities and Applications},
      YEAR = {2025},
     PAGES = {Paper No. 103, 32},
      ISSN = {1029-242X},
   MRCLASS = {52A40},
  MRNUMBER = {4951555},
MRREVIEWER = {Xiao\ Li},
       DOI = {10.1186/s13660-025-03356-w},
       URL = {https://doi.org/10.1186/s13660-025-03356-w},
}

@article {meyer3,
    AUTHOR = {Meyer, Mark},
     TITLE = {The {D}yn-{F}arkhi conjecture and the convex hull of a sumset
              in two dimensions},
   JOURNAL = {Proc. Amer. Math. Soc.},
  FJOURNAL = {Proceedings of the American Mathematical Society},
    VOLUME = {153},
      YEAR = {2025},
    NUMBER = {11},
     PAGES = {4889--4901},
      ISSN = {0002-9939,1088-6826},
   MRCLASS = {52A40 (52A10)},
  MRNUMBER = {4971580},
MRREVIEWER = {Tomasz\ Tkocz},
       DOI = {10.1090/proc/17313},
       URL = {https://doi.org/10.1090/proc/17313},
}

@book {schneider-2014,
    AUTHOR = {Schneider, Rolf},
     TITLE = {Convex bodies: the {B}runn-{M}inkowski theory},
    SERIES = {Encyclopedia of Mathematics and its Applications},
    VOLUME = {151},
   EDITION = {expanded},
 PUBLISHER = {Cambridge University Press, Cambridge},
      YEAR = {2014},
     PAGES = {xxii+736},
      ISBN = {978-1-107-60101-7},
   MRCLASS = {52-02 (52A20 52A39)},
  MRNUMBER = {3155183},
MRREVIEWER = {Andrea\ Colesanti},
}

@article{starr1,
  title={Quasi-equilibria in markets with non-convex preferences},
  author={Starr, Ross M},
  journal={Econometrica: journal of the Econometric Society},
  pages={25--38},
  year={1969},
  publisher={JSTOR}
}

@article{vanhintum3,
  title={The sharp threshold for {H}ausdorff convexification under {M}inkowski addition},
  author={van Hintum, Peter},
  journal={arXiv preprint arXiv:2606.10815},
  year={2026}
}
\bibliographystyle{plain}

\end{document}